\documentclass[11pt,a4paper]{article}

\usepackage[utf8]{inputenc}
\usepackage{amsmath, amssymb, amsthm}
\usepackage{hyperref}
\usepackage{geometry}
\newtheorem{theorem}{Theorem}[section]
\newtheorem{lemma}[theorem]{Lemma}
\newtheorem{corollary}[theorem]{Corollary}
\newtheorem{proposition}[theorem]{Proposition}
\theoremstyle{definition}
\newtheorem{definition}[theorem]{Definition}
\newtheorem{question}{Open Problem}
\newtheorem{remark}[theorem]{Remark}

\title{$m$-Rigidity, Finite-One, and Bounded Finite-One Degrees \\ Inside Typical Many-One Degrees}

\author{
Patrizio Cintioli\\
Mathematics Division, School of Science and Technology\\
University of Camerino, Italy\\
\texttt{patrizio.cintioli@unicam.it}
}

\date{} % niente data

\begin{document}

\maketitle

\begin{abstract}
In recent work, the notion of $m$-rigidity was introduced as a sufficient condition for the
existence of infinite antichains of $1$-degrees inside many-one degrees. Motivated by a recent
preprint of Richter, Stephan, and Zhang on intermediate reducibilities inside many-one degrees, we
study the finite-one and bounded finite-one structure of the many-one degree of an $m$-rigid set.

First, combining bi-immunity of $m$-rigid sets with a theorem of Richter, Stephan, and Zhang,
we show that for Lebesgue-almost every set $A$, and for a comeager class of sets $A$, the
many-one degree $\deg_m(A)$ contains a least finite-one degree. Second, we prove that if
$A$ is $m$-rigid, then $\deg_m(A)$ contains infinitely many pairwise incomparable finite-one
degrees. Third, we show that the arithmetic thickenings of an $m$-rigid set form a strict ascending chain
\[
A_{(1)} <_1 A_{(2)} <_1 \cdots
\]
of $1$-degrees entirely contained within a single bounded finite-one degree. Finally, by introducing bounded calibrated domains, we prove that the bounded finite-one degree of an $m$-rigid set also contains an infinite antichain of $1$-degrees, precluding any linear order.

These results yield almost-sure and comeager partial answers to all three open problems
posed by Richter, Stephan, and Zhang.
\end{abstract}

\noindent\textbf{2020 Mathematics Subject Classification.}
Primary 03D30; Secondary 03D32.

\medskip

\noindent\textbf{Keywords.}
Many-one degrees; finite-one degrees; bounded finite-one degrees; one-one degrees; m-rigidity; bi-immunity; Martin-L\"of randomness.

\section{Introduction}

Many-one reducibility and its refinements by one-one, bounded finite-one, and finite-one reductions give rise to
nested degree structures whose mutual interaction is still only partially understood. For a set
$A \subseteq \omega$, the many-one degree $\deg_m(A)$ may split into several finite-one
degrees, which may split into bounded finite-one degrees, and each of these may in turn split into several one-one degrees. While the
relation between one-one degrees and many-one degrees has been studied since the beginning
of recursion theory, the intermediate reducibilities remain comparatively less explored.

A central question in this area goes back to Odifreddi's Problem~5, which asks whether every
nonrecursive nonirreducible many-one degree contains an infinite antichain of one-one degrees
\cite{Odifreddi1981,Odifreddi1999}. In \cite{Cintioli2026}, a sufficient condition for such
antichains was isolated: a set $A$ is called \emph{$m$-rigid} if every total computable
$m$-autoreduction of $A$ is eventually the identity. It was shown there that every $m$-rigid
set has an infinite antichain of $1$-degrees inside its many-one degree, and that the class of
$m$-rigid sets has Lebesgue measure $1$ and is comeager.

In a recently submitted preprint, Richter, Stephan, and Zhang \cite{RSZ2026} study chains and
antichains inside many-one degrees and related intermediate reducibilities. 
In particular, they generalise Batyrshin's theorem to all nonrecursive and nonirreducible many-one degrees, 
thereby answering Odifreddi's Open Problem~5 \cite{Odifreddi1981,Odifreddi1999} affirmatively.
Their manuscript also extensively investigates finite-one and bounded finite-one degrees inside many-one degrees and ends
with three open problems:
\begin{enumerate}
    \item Does every nonrecursive many-one degree contain a least finite-one degree?
    \item Are there nonrecursive many-one degrees consisting of at least two but at most finitely many finite-one degrees?
    \item Are there bounded finite-one degrees consisting exactly of a dense linearly ordered set of one-one degrees?
\end{enumerate}

The purpose of the present paper is to combine the rigidity viewpoint of \cite{Cintioli2026}
with these new questions on intermediate reducibilities. We do not settle the three problems in full
generality. Instead, we show that all of them admit almost-sure and comeager partial
answers on the class of $m$-rigid sets. Since this class has Lebesgue measure $1$ and is
comeager, these results describe the fine structure of \emph{typical} many-one degrees.

Our main results are as follows:

\begin{itemize}
\item For every $m$-rigid set $A$, the many-one degree $\deg_m(A)$ contains a least
finite-one degree (Section 3).

\item For every $m$-rigid set $A$, a single bounded finite-one degree inside $\deg_m(A)$ already
contains an infinite strict ascending chain of one-one degrees, obtained via arithmetic thickenings (Section 4).

\item For every $m$-rigid set $A$, the many-one degree $\deg_m(A)$ contains infinitely many
pairwise incomparable finite-one degrees, preventing any finite collapse (Section 5). 

\item By restricting our calibrated-domain construction, we prove that a single bounded finite-one degree of an $m$-rigid set contains an infinite antichain of $1$-degrees, proving that typical bounded finite-one degrees are never linearly ordered (Section 6; compare \cite[Remark~19]{RSZ2026} for the case of Martin-L\"of random sets).
\end{itemize}

Thus, within the typical part of Cantor space, the many-one degree possesses a least
finite-one degree, but moving upwards the structure fractures infinitely at both the finite-one and bounded finite-one levels.
Any set witnessing a negative answer to Open Problem~1 or the structural phenomena asked about in Open Problems~2 and~3 must lie outside the class of $m$-rigid sets, and therefore inside a null and meager set.

\section{Preliminaries}

This section collects the basic notation and recalls the two inputs from
\cite{Cintioli2026} that will be used throughout the paper: the bi-immunity of $m$-rigid
sets and the fact that $m$-rigidity holds on a measure-$1$ and comeager class.
We assume standard notation from computability theory. 
For general references see, e.g.,
Odifreddi~\cite{OdifreddiCRT} and Soare~\cite{Soare1987}.
Let $\omega$ denote the set of natural numbers, and $\langle \cdot, \cdot \rangle : \omega \times \omega \to \omega$ be a standard computable pairing function with computable projections $\pi_1$ and $\pi_2$.

\begin{definition}[Reducibilities and degrees]
Let $A,B \subseteq \omega$.
\begin{itemize}
\item $A \le_m B$ if there exists a total computable function $f$ such that
$\forall x\,[\,x \in A \iff f(x) \in B\,]$.
\item $A \le_1 B$ if $A \le_m B$ via an injective computable function $f$.
\item $A \le_{bfin} B$ if $A \le_m B$ via a finite-to-one computable function $f$ bounded by a global constant $c \ge 1$, i.e., $\forall y\,|f^{-1}(y)| \le c$.
\item $A \le_{fin} B$ if $A \le_m B$ via a finite-to-one computable function $f$,
i.e.\ $\forall y\,|f^{-1}(y)|<\infty$.
\end{itemize}
For $r\in\{m,1,bfin,fin\}$ we write $A \equiv_r B$ if $A \le_r B$ and $B \le_r A$.
The \emph{$r$-degree} of $A$ is
\[
\deg_r(A)=\{\,B\subseteq\omega : B \equiv_r A\,\}.
\]
We write $A<_r B$ for $A\le_r B$ and $B\not\le_r A$, and we order $r$-degrees by
$\deg_r(A)\le_r \deg_r(B)$ iff $A\le_r B$.
\end{definition}

\begin{remark}
Clearly, $\le_1\Rightarrow\le_{bfin}\Rightarrow\le_{fin}\Rightarrow\le_m$.
When we speak of ``finite-one degrees inside $\deg_m(A)$'' we mean the family
$\{\deg_{fin}(B): B \equiv_m A\}$ ordered by $\le_{fin}$.
\end{remark}

\begin{definition}[\cite{Cintioli2026}]
A set $A \subseteq \omega$ is \emph{$m$-rigid} if for every total computable function $f$ such that $x \in A \iff f(x) \in A$, we have $f(x) = x$ for almost all $x$.
\end{definition}

A fundamental consequence of $m$-rigidity is that it implies bi-immunity, hence in particular
neither $A$ nor $\overline{A}$ contains an infinite c.e.\ subset.

\begin{lemma}[Cintioli {\cite[Lemma 2.4]{Cintioli2026}}] \label{lem:rigid_biimmune}
Every $m$-rigid set is bi-immune.
\end{lemma}

It is a known result that both 1-generic sets and Martin-L\"of random sets are $m$-rigid \cite{Cintioli2026}. Consequently, the class of $m$-rigid sets has Lebesgue measure 1 and is comeager in the Cantor space.

\begin{remark}
Throughout the paper we use the standard notion of immunity for infinite sets:
an infinite set is \emph{immune} if it has no infinite c.e.\ subset, equivalently if it contains no infinite computable subset.
Accordingly, \emph{bi-immune} means that both $A$ and $\overline{A}$ are immune in this sense.
\end{remark}

Before proceeding to our main results, it is conceptually insightful to observe how $m$-rigidity
interacts with the classical hierarchy of intermediate reducibilities discussed above.
From the definitions one immediately has
\[
\le_1 \Rightarrow \le_{bfin} \Rightarrow \le_{fin} \Rightarrow \le_m .
\]
The substantive issue is not the formal containment of these reducibilities, but how much of
this hierarchy is reflected at the level of degree structures inside a fixed many-one degree:
depending on the set, intermediate notions may collapse, or they may induce genuinely distinct
strata of degrees.

For computably enumerable sets, the first systematic separations involving finite-one and
related bounded variants go back to Maslova~\cite{Maslova1979}.
More recently, Richter, Stephan, and Zhang~\cite{RSZ2026} analysed the fine internal structure
of finite-one and bounded finite-one degrees inside many-one degrees for arbitrary sets,
exhibiting, for example, a nonrecursive bounded finite-one degree whose $1$-degrees form a
strict $\omega$-chain \cite[Theorem~18]{RSZ2026}, and showing that every nonrecursive
nonirreducible finite-one degree contains an antichain of bounded finite-one degrees
\cite[Remark~20]{RSZ2026}.

At the opposite extreme, for highly structured sets this hierarchy can collapse completely:
by Myhill's isomorphism theorem, every set many-one equivalent to the Halting Problem is
computably isomorphic to it; hence, its many-one degree collapses into a single one-one degree.

By contrast, for $m$-rigid sets the hierarchy is already strict \emph{within} a single many-one
degree: Proposition~\ref{prop:strict_hierarchy} shows that each implication above is proper
inside $\deg_m(A)$ for every $m$-rigid set $A$. We prove this after introducing arithmetic
thickenings and establishing Theorem~\ref{thm:strict_chain}; see Section~4.

\section{A Least Finite-One Degree for Typical Sets}

We first address the general existence of a least finite-one degree inside a many-one degree, which directly partially answers Open Problem 1 raised by Richter, Stephan, and Zhang.

\begin{question}[Richter, Stephan, Zhang \cite{RSZ2026}]
Does every nonrecursive many-one degree contain a least finite-one degree?
\end{question}

In their paper, they established a profound connection between bi-immunity and the lower bound of finite-one reducibility within an $m$-degree.

\begin{theorem}[Richter, Stephan, Zhang {\cite[Theorem 16(4)]{RSZ2026}}] \label{thm:rsz_16.4}
If $A$ is a bi-immune set, then the many-one degree of $A$ contains a least finite-one degree.
\end{theorem}

By naturally combining this theorem with the structural properties of $m$-rigid sets, we obtain a straightforward but powerful measure-theoretic consequence.

\begin{corollary} \label{cor:prob1_answer}
For Lebesgue-almost every set $A$ and in the sense of Baire category (comeager), the many-one degree $\deg_m(A)$ contains a least $fin$-degree. Thus, Open Problem 1 of \cite{RSZ2026} has an affirmative answer for typical sets.
\end{corollary}
\begin{proof}
As established in \cite{Cintioli2026}, the class of $m$-rigid sets has Lebesgue measure 1 and is comeager. By Lemma \ref{lem:rigid_biimmune}, every $m$-rigid set is bi-immune. Therefore, by Theorem \ref{thm:rsz_16.4}, every $m$-rigid set belongs to an $m$-degree that contains a least $fin$-degree. Any counterexample lacking a least $fin$-degree must therefore be non-$m$-rigid.
\end{proof}

\section{Infinite Chains of 1-Degrees inside a Bounded Finite-One Degree}

Having established in the previous section that typical sets possess a least $fin$-degree, we now investigate the internal complexity of individual intermediate degrees. Arithmetic thickenings provide a uniform family of representatives whose one-one degrees form a strict ascending chain. We will show that this infinite chain lives entirely inside a \emph{single bounded finite-one degree}.
Richter, Stephan, and Zhang constructed an example of a nonrecursive bounded finite-one degree
whose $1$-degrees form a strict $\omega$-chain \cite[Theorem~18]{RSZ2026}.
Here we show that an $\omega$-chain of $1$-degrees already occurs uniformly inside $\deg_{bfin}(A)$
for every $m$-rigid set $A$, via arithmetic thickenings.
To achieve this, we introduce the contiguous arithmetic block thickening of a set. Standard cylinder constructions (e.g., $A \times \omega$) introduce infinite computable sets in their complements, inherently preventing finite-one reductions to bi-immune sets. Our block thickening avoids this issue and is naturally adapted to bounded finite-one reducibility.

\begin{definition}
For any set $A$ and any integer $k \geq 1$, the \emph{$k$-thickening} of $A$ is defined as:
$$ A_{(k)} = \{ x \in \omega \mid \lfloor x / k \rfloor \in A \} $$
\end{definition}

\begin{lemma} \label{lem:bfin_equiv}
For any set $A$ and any integer $k \geq 1$, $A \equiv_{bfin} A_{(k)}$.
\end{lemma}
\begin{proof}
We have $A \leq_1 A_{(k)}$ via the injective computable function $g(x) = kx$. Since injectivity implies a bound of $c=1$, we automatically have $A \leq_{bfin} A_{(k)}$.

Conversely, $A_{(k)} \leq_{bfin} A$ via $f(x) = \lfloor x / k \rfloor$. The function $f$ maps exactly $k$ consecutive integers (the set $\{ky, ky+1, \dots, ky+k-1\}$) to a single value $y$. Therefore, $f$ is strictly bounded finite-one (exactly $k$-to-$1$), with global bound $c=k$, and $x \in A_{(k)} \iff \lfloor x / k \rfloor \in A$.
Thus, $A \equiv_{bfin} A_{(k)}$.
\end{proof}

Although Lemma~\ref{lem:bfin_equiv} shows that all thickenings $A_{(k)}$ lie in the exact same bounded finite-one degree, for an $m$-rigid set $A$ their $1$-degrees form a strict ascending chain.

\begin{theorem} \label{thm:strict_chain}
If $A$ is an $m$-rigid set, then for all $k \ge 1$, $A_{(k)} <_1 A_{(k+1)}$.
\end{theorem}
\begin{proof}
First, we show $A_{(k)} \leq_1 A_{(k+1)}$. Consider the function $p(x) = (k+1) \lfloor x / k \rfloor + (x \bmod k)$. 
The function $p$ is strictly increasing, thus injective. Moreover, since $(x \bmod k) < k < k+1$, we have $\lfloor p(x) / (k+1) \rfloor = \lfloor x / k \rfloor$.
Therefore, $x \in A_{(k)} \iff \lfloor x / k \rfloor \in A \iff \lfloor p(x) / (k+1) \rfloor \in A \iff p(x) \in A_{(k+1)}$. Hence $A_{(k)} \leq_1 A_{(k+1)}$.

To show the reduction is strict, assume for a contradiction that $A_{(k+1)} \leq_1 A_{(k)}$ via some injective computable function $h$.
Thus, $x \in A_{(k+1)} \iff h(x) \in A_{(k)}$, which implies $\lfloor x / (k+1) \rfloor \in A \iff \lfloor h(x) / k \rfloor \in A$.
For each $j \in \{0, \dots, k\}$, define the computable function $f_j(y) = \lfloor h((k+1)y + j) / k \rfloor$.
Notice that $\lfloor ((k+1)y + j) / (k+1) \rfloor = y$. Therefore, $y \in A \iff f_j(y) \in A$.
This means each $f_j$ is a valid $m$-autoreduction of $A$.

Because $A$ is $m$-rigid, for each $j$, $f_j(y) = y$ for almost all $y$. 
Since there are only $k+1$ functions $f_j$, we can take $N$ larger than all their finitely many exceptional points.
Thus, there exists an $N \in \omega$ such that for all $y \geq N$ and all $j \in \{0, \dots, k\}$, $f_j(y) = y$.
Fix a $y \geq N$. The definition of $f_j$ yields $\lfloor h((k+1)y + j) / k \rfloor = y$.
This implies that for each $j$, the value $h((k+1)y + j)$ belongs to the set $V_y = \{ ky, \dots, ky+k-1 \}$, which has size $k$.
However, we are evaluating the injective function $h$ on $k+1$ distinct inputs. Since $h$ maps $k+1$ distinct elements into $V_y$ of size $k$, the Pigeonhole Principle dictates that $h$ cannot be injective, a contradiction. Thus, $A_{(k+1)} \not\leq_1 A_{(k)}$.
\end{proof}

\begin{corollary} \label{cor:chain_bfin}
The bounded finite-one degree of any $m$-rigid set does not collapse into a single $1$-degree, as it contains an infinite strict ascending chain of $1$-degrees.
\end{corollary}

\medskip

We now provide the proof of the strictness phenomenon announced in Section~2:
for an $m$-rigid set, the classical reducibility hierarchy is already strict inside a single many-one degree.

\begin{proposition} \label{prop:strict_hierarchy}
If $A$ is an $m$-rigid set, the implications $\le_1 \implies \le_{bfin} \implies \le_{fin} \implies \le_m$
are all strict inside the many-one degree of $A$.
Specifically, there exist sets $B, C, D \equiv_m A$ such that:
\begin{enumerate}
    \item $B \le_{bfin} A$ but $B \not\le_1 A$.
    \item $C \le_{fin} A$ but $C \not\le_{bfin} A$.
    \item $D \le_m A$ but $D \not\le_{fin} A$.
\end{enumerate}
\end{proposition}

\begin{proof}
We give explicit witnesses.
\begin{enumerate}
    \item Let $B = A_{(2)}$. By Lemma~\ref{lem:bfin_equiv} we have $B \le_{bfin} A$ and $B \equiv_m A$.
    Moreover, Theorem~\ref{thm:strict_chain} with $k=1$ yields $A_{(1)} <_1 A_{(2)}$, i.e.\ $A<_1 B$.
    Hence $B \not\le_1 A$.

    \item Let $D_{pyr} = \{ \langle x, y \rangle \mid y \le x \}$ be a ``pyramid'' domain, and fix a computable bijection $\rho: \omega \to D_{pyr}$.
    Define
    \[
    C = \{ n \in \omega \mid \pi_1(\rho(n)) \in A \}.
    \]
    Then $C \le_{fin} A$ via $q(n) = \pi_1(\rho(n))$, because the preimage of any $x$ corresponds exactly to a column of size $x+1$, which is finite.
    
    Conversely, $A\le_m C$ via $r(x)=\rho^{-1}(\langle x,0\rangle)$: in fact, observe that $\rho(r(x))=\langle x,0\rangle$, hence $\pi_1(\rho(r(x)))=x$; moreover, since $0\le x$ implies $\langle x,0\rangle\in D_{pyr}$ it follows that   
    \[
    x \in A \iff \pi_1(\langle x,0\rangle) \in A \iff r(x)\in C.
    \]
    Hence $C\equiv_m A$.

    Assume for a contradiction that $C \le_{bfin} A$ via some computable function $h$ bounded by a global constant $c \ge 1$.
    This induces a map $f: D_{pyr} \to \omega$, defined by $f(z) = h(\rho^{-1}(z))$, which is also bounded by $c$ and satisfies
    \[
    \pi_1(z) \in A \iff f(z) \in A.
    \]
    For any $x \ge c$, the column $P_x = \{ \langle x, y \rangle \mid y \le x \}$ has size $x+1 > c$.
    Thus $f$ cannot map all elements of $P_x$ to the single value $x$, so there exists some $y \le x$ such that $f(\langle x, y \rangle) \neq x$.

    Define a total computable function $g: \omega \to \omega$ as follows: for $x < c$, let $g(x)=x$; for $x \ge c$, search for the least $y \le x$ such that $f(\langle x, y \rangle) \neq x$, and set $g(x)=f(\langle x, y \rangle)$.
    Since $\pi_1(\langle x, y \rangle)=x$, we have $x \in A \iff g(x)\in A$, so $g$ is an $m$-autoreduction of $A$.
    But $g(x)\neq x$ for all $x \ge c$, contradicting the $m$-rigidity of $A$.
    Thus, $C \not\le_{bfin} A$.

    \item Let $D = A \times \omega$. Clearly $D \equiv_m A$, and $D \le_m A$ via $f(\langle x, y \rangle) = x$.
    Assume for a contradiction that $D \le_{fin} A$ via some finite-to-one computable function $h$.
    For any $x$, the infinite column $V_x = \{ \langle x, y \rangle \mid y \in \omega \}$ must be mapped by $h$ to an infinite c.e.\ set $W_x$ (if $W_x$ were finite, some element in $W_x$ would have infinitely many preimages in $V_x$, contradicting that $h$ is finite-to-one).
    Because $h$ preserves membership, if $x \in A$ then $W_x \subseteq A$, and if $x \notin A$ then $W_x \subseteq \overline{A}$.
    This means either $A$ or $\overline{A}$ contains an infinite c.e.\ set, contradicting the bi-immunity of $A$ (Lemma~\ref{lem:rigid_biimmune}).
    Thus, $D \not\le_{fin} A$.
\end{enumerate}
\end{proof}

\section{Infinite Antichains of Finite-One Degrees}

This section contains the main structural separation construction. Using calibrated domains, we
convert separation between computable parameters into finite-one incomparability inside a
fixed many-one degree. This provides the tool to answer Open Problem 2 of \cite{RSZ2026}.

\begin{definition}\label{def:calibrated_domain}
Let $S \subseteq \omega$ be a computable set. We define its \emph{calibrated domain} $D_S \subseteq \omega$ as:
$$ D_S = \{ \langle x, i \rangle \mid x \in S \text{ or } i = 0 \}. $$
\end{definition}

Notice that $D_S$ is an infinite computable set. We fix a standard computable bijection $\sigma_S : \omega \to D_S$. For any set $A\subseteq\omega$, we define the set $B_S \subseteq \omega$ as:
$$ B_S = \{ n \in \omega \mid \pi_1(\sigma_S(n)) \in A \} .$$

\begin{lemma} \label{lem:BS_equiv_A}
For any set $A\subseteq\omega$ and any computable set $S$, we have $B_S \equiv_m A$.
\end{lemma}
\begin{proof}
It is clear that $B_S \leq_m A$ via the computable function $q(n) = \pi_1(\sigma_S(n))$.
Conversely, $A \leq_m B_S$ via $r(x) = \sigma_S^{-1}(\langle x, 0 \rangle)$. 
Note that $\langle x,0\rangle\in D_S$ for every $x$, hence $r$ is total computable (by effective search).
Since $\sigma_S(r(x))=\langle x,0\rangle$, we have 
\[
r(x)\in B_S \iff \pi_1(\sigma_S(r(x)))\in A \iff \pi_1(\langle x,0\rangle)\in A \iff x\in A.
\]
\end{proof}

We now prove that if $T$ contains infinitely many elements not in $S$, a finite-one reduction cannot embed $B_T$ into $B_S$.

\begin{theorem} \label{thm:not_fin_reducible}
Let $A$ be an $m$-rigid set. If $S$ and $T$ are computable sets such that $T \setminus S$ is infinite, then $B_T \not\leq_{fin} B_S$.
\end{theorem}
\begin{proof}
Fix an $m$-rigid set $A$ and assume for a contradiction that $B_T \leq_{fin} B_S$ via a computable finite-to-one function $h$.
This induces a finite-to-one computable function on the domains, $f: D_T \to D_S$, defined by $f(z) = \sigma_S(h(\sigma_T^{-1}(z)))$.
By the validity of the reduction, for any $z \in D_T$, we have $\pi_1(z) \in A \iff \pi_1(f(z)) \in A$.

For each $x \in T \setminus S$, the element $x$ generates an infinite column in $D_T$:
$$ C_x = \{ \langle x, i \rangle \mid i \in \omega \} \subseteq D_T $$
Since $f$ is finite-to-one, the image $f(C_x)$ must be an infinite computably enumerable subset of $D_S$. Let $V_x = \{ \pi_1(w) \mid w \in f(C_x) \}$ be the projection of this image onto the first coordinate.
For every $y \in V_x$, there exists some $i \in \omega$ such that $\pi_1(f(\langle x, i \rangle)) = y$. Since the reduction preserves membership relative to $A$, $y \in A \iff x \in A$. 
If $V_x$ were infinite, it would constitute an infinite c.e.\ subset of $A$ (if $x \in A$) or of $\overline{A}$ (if $x \notin A$). However, by Lemma \ref{lem:rigid_biimmune}, $A$ is bi-immune, which strictly prohibits this. Thus, $V_x$ must be a finite set.

Because $V_x$ is finite but $f(C_x)$ is infinite, $f(C_x)$ must contain infinitely many pairs with the same first coordinate $y \in V_x$. By the definition of $D_S$, the only elements that possess infinitely many second coordinates are precisely those $y \in S$. Consequently, $V_x$ must contain at least one element belonging to $S$. In other words, $V_x \cap S \neq \emptyset$.

We exploit this to define a total computable function $g: \omega \to \omega$ as follows:
\begin{itemize}
    \item If $x \in T \setminus S$, we algorithmically compute $f(\langle x, i \rangle)$ for $i = 0, 1, 2, \dots$ until we find a pair whose first coordinate belongs to $S$. Because $V_x \cap S \neq \emptyset$ and $S$ is computable, this search is guaranteed to terminate. We set $g(x) = \pi_1(f(\langle x, i \rangle))$.
    \item If $x \notin T \setminus S$, we set $g(x) = x$.
\end{itemize}

We claim $g$ is an $m$-autoreduction of $A$. If $x \notin T \setminus S$, $g(x) = x$, trivially preserving membership. If $x \in T \setminus S$, $g(x) = y = \pi_1(f(\langle x, i \rangle))$. As shown earlier, $\pi_1(f(z)) \in A \iff \pi_1(z) \in A$. Here $\pi_1(z) = x$, so $y \in A \iff x \in A$. Thus, $x \in A \iff g(x) \in A$.

Since $A$ is $m$-rigid, $g$ must be the identity almost everywhere. However, for all $x \in T \setminus S$, the search guarantees that $g(x) \in S$. Since $S \cap (T \setminus S) = \emptyset$, we have $g(x) \neq x$ for all $x \in T \setminus S$. Because $T \setminus S$ is infinite, $g(x) \neq x$ for infinitely many $x$. This is a direct contradiction of the $m$-rigidity of $A$. Hence, $B_T \not\leq_{fin} B_S$.
\end{proof}

\begin{corollary} \label{cor:infinite_antichain}
The $m$-degree of any $m$-rigid set contains an infinite antichain of finite-one degrees.
\end{corollary}
\begin{proof}
Let $\{S_k\}_{k \in \omega}$ be a family of mutually disjoint infinite computable sets. For any $i \neq j$, $S_i \cap S_j = \emptyset$, meaning $S_i \setminus S_j = S_i$, which is infinite. By Theorem \ref{thm:not_fin_reducible}, $B_{S_i} \not\leq_{fin} B_{S_j}$ and $B_{S_j} \not\leq_{fin} B_{S_i}$. 
The sequence $\{B_{S_k}\}_{k \in \omega}$ thus forms an infinite antichain of finite-one degrees inside the $m$-degree of $A$, since each $B_{S_k}$ belongs to $\deg_m(A)$ by Lemma \ref{lem:BS_equiv_A}.
\end{proof}

The existence of an infinite antichain enforces a strict constraint on the possible finite cardinalities of $fin$-degrees within a given $m$-degree.

\begin{question}[Richter, Stephan, Zhang \cite{RSZ2026}]
Are there nonrecursive many-one degrees consisting of at least two but at most finitely many finite-one degrees (fin-degrees)?
\end{question}

\begin{corollary} \label{cor:prob2_answer}
If the $m$-degree of $A$ consists of at least two but at most finitely many finite-one degrees, then $A$ is not $m$-rigid. Therefore, for Lebesgue-almost every set $A$ and in the sense of Baire category, the many-one degree $\deg_m(A)$ contains infinitely many pairwise incomparable $fin$-degrees.
\end{corollary}

\begin{proof}
If $A$ were $m$-rigid, Corollary~\ref{cor:infinite_antichain} would yield infinitely many pairwise incomparable $fin$-degrees inside $\deg_m(A)$, contradicting the hypothesis.
The measure and category statement follows because $m$-rigid sets form a measure-$1$ and comeager class.
\end{proof}

\section{Bounded Finite-One Degrees and Open Problem 3}

We finally address the third open problem, which focuses on the strict substructure of bounded finite-one degrees.

\begin{question}[Richter, Stephan, Zhang \cite{RSZ2026}]
Are there bounded finite-one degrees consisting exactly of a dense linearly ordered set of one-one degrees?
\end{question}

For a bounded finite-one degree to consist \emph{exactly} of a linearly ordered set, it must fundamentally lack any incomparable pair of $1$-degrees. We prove that typical degrees violate this by restricting the column size in our calibrated domains to achieve a global bounded finite-one equivalence, while constructing an infinite antichain of $1$-degrees.
We note that Richter, Stephan, and Zhang observed that if $A$ is Martin-L\"of random,
then $\deg_{bfin}(A)$ contains an infinite antichain of $1$-degrees \cite[Remark~19]{RSZ2026}.
Our argument below strengthens this for typicality in the sense of Baire category,
by proving the existence of such an antichain for every $m$-rigid set via bounded calibrated domains.
\begin{definition}\label{def:bounded_calibrated_domain}
Let $S \subseteq \omega$ be a computable set. We define its \emph{bounded calibrated domain} $E_S \subseteq \omega$ as:
$$ E_S = \{ \langle x, 0 \rangle \mid x \in \omega \} \cup \{ \langle x, 1 \rangle \mid x \in S \}. $$
\end{definition}

Notice that $E_S$ is a computable set where every $x \in \omega$ has exactly two copies if $x \in S$, and exactly one copy if $x \notin S$. We fix a standard computable bijection $\tau_S : \omega \to E_S$. For any set $A\subseteq\omega$, we define the set $C_S \subseteq \omega$ as:
$$ C_S = \{ n \in \omega \mid \pi_1(\tau_S(n)) \in A \} .$$

\begin{lemma} \label{lem:CS_equiv_A_bfin}
For any set $A\subseteq\omega$ and any computable set $S$, we have $C_S \equiv_{bfin} A$.
\end{lemma}
\begin{proof}
It is clear that $C_S \leq_m A$ via the computable function $q(n) = \pi_1(\tau_S(n))$. Since every element $x \in \omega$ has at most two copies in $E_S$, the function $q$ is strictly bounded finite-to-one (bounded by $c=2$). Thus $C_S \le_{bfin} A$.

Conversely, $A \leq_1 C_S$ via $r(x) = \tau_S^{-1}(\langle x, 0 \rangle)$. Since $\langle x, 0 \rangle \in E_S$ for every $x$, $r$ is total, computable, and injective. Because $\le_1$ implies $\le_{bfin}$, we have $A \le_{bfin} C_S$. Thus, $C_S \equiv_{bfin} A$.
\end{proof}

We now show that infinite differences between the computable index sets preclude $1$-reducibility.

\begin{theorem} \label{thm:not_1_reducible_bfin}
Let $A$ be an $m$-rigid set. If $S$ and $T$ are computable sets such that $T \setminus S$ is infinite, then $C_T \not\leq_1 C_S$.
\end{theorem}
\begin{proof}
Fix an $m$-rigid set $A$ and assume for a contradiction that $C_T \leq_1 C_S$ via an injective computable function $h$.
This induces an injective computable function on the domains, $f: E_T \to E_S$, defined by $f(z) = \tau_S(h(\tau_T^{-1}(z)))$.
By the validity of the reduction, for any $z \in E_T$, we have $\pi_1(z) \in A \iff \pi_1(f(z)) \in A$.

We define two total computable functions $g_0, g_1: \omega \to \omega$ as follows:
\begin{itemize}
    \item $g_0(x) = \pi_1(f(\langle x, 0 \rangle))$ for all $x \in \omega$. (Well-defined since $\langle x, 0 \rangle \in E_T$ for all $x$).
    \item $g_1(x) = \pi_1(f(\langle x, 1 \rangle))$ if $x \in T$, and $g_1(x) = x$ if $x \notin T$. (Well-defined and computable since $T$ is a computable set).
\end{itemize}

Since $f$ preserves membership relative to $A$, $x \in A \iff \pi_1(\langle x, 0 \rangle) \in A \iff \pi_1(f(\langle x, 0 \rangle)) \in A \iff g_0(x) \in A$. Thus, $g_0$ is a valid $m$-autoreduction of $A$.
Similarly, for $x \in T$, $x \in A \iff \pi_1(\langle x, 1 \rangle) \in A \iff \pi_1(f(\langle x, 1 \rangle)) \in A \iff g_1(x) \in A$. For $x \notin T$, $g_1(x) = x$, trivially preserving membership. Thus, $g_1$ is also a valid $m$-autoreduction of $A$.

Because $A$ is $m$-rigid, both $g_0$ and $g_1$ must be the identity almost everywhere. Hence, there exists an $N \in \omega$ such that for all $x \ge N$, $g_0(x) = x$ and $g_1(x) = x$.

Since $T \setminus S$ is infinite, we can choose an $x \ge N$ such that $x \in T \setminus S$. 
Because $x \ge N$, we have $g_0(x) = x$, meaning that $f(\langle x, 0 \rangle) \in E_S$ has first coordinate $x$. Since $x \notin S$, the definition of $E_S$ dictates that the \emph{only} element in $E_S$ with first coordinate $x$ is $\langle x, 0 \rangle$. Therefore, we must have $f(\langle x, 0 \rangle) = \langle x, 0 \rangle$.

Simultaneously, since $x \in T$ and $x \ge N$, we have $g_1(x) = x$, meaning that $f(\langle x, 1 \rangle) \in E_S$ also has first coordinate $x$. By the exact same logic, we must have $f(\langle x, 1 \rangle) = \langle x, 0 \rangle$.

This implies $f(\langle x, 0 \rangle) = f(\langle x, 1 \rangle) = \langle x, 0 \rangle$. However, $\langle x, 0 \rangle$ and $\langle x, 1 \rangle$ are distinct elements in $E_T$. This directly contradicts the injectivity of $f$ (and thus of $h$). Hence, $C_T \not\leq_1 C_S$.
\end{proof}

\begin{corollary} \label{cor:prob3_answer}
For Lebesgue-almost every set $A$ and in the sense of Baire category (comeager), the bounded finite-one degree $\deg_{bfin}(A)$ is not linearly ordered. Thus, Open Problem 3 of \cite{RSZ2026} has a negative answer for typical sets.
\end{corollary}
\begin{proof}
Let $\{S_k\}_{k \in \omega}$ be a family of mutually disjoint infinite computable sets. By Theorem \ref{thm:not_1_reducible_bfin}, $C_{S_i} \not\leq_1 C_{S_j}$ and $C_{S_j} \not\leq_1 C_{S_i}$. Since all $C_{S_k}$ belong to $\deg_{bfin}(A)$, the bounded finite-one degree of any $m$-rigid set contains an infinite antichain of $1$-degrees. A partially ordered set containing an antichain of size $\ge 2$ cannot be linearly ordered. Since typical sets are $m$-rigid, typical bounded finite-one degrees are never linearly ordered.
\end{proof}

\section{Conclusion}

We conclude by summarising the fine-structure picture that emerges from $m$-rigidity and by
relating it back to the three open problems of Richter, Stephan, and Zhang \cite{RSZ2026}.
By generalising the application of $m$-rigidity, we have uncovered a deeply fractal-like structure within typical many-one degrees. 

For almost all sets (and, in fact, for a comeager class of sets):
\begin{enumerate}
    \item The $m$-degree admits a least finite-one degree (Open Problem 1).
    \item The $m$-degree contains an infinite antichain of finite-one degrees, preventing it from consisting of only finitely many finite-one degrees (Open Problem 2).
    \item A single bounded finite-one degree already contains both an infinite strict ascending chain and an infinite antichain of $1$-degrees, precluding any possible linear order (Open Problem 3).
\end{enumerate}

This yields a clear measure-theoretic and category-theoretic picture of the fine structure
induced by finite-one, bounded finite-one, and one-one reducibility inside a typical many-one degree.
Consequently, the class of potential counterexamples to Open Problem~1 and of sets exhibiting the
configurations sought in Open Problems~2 and~3 is contained in a null and meager set.

\section*{Acknowledgments}

This work is the result of an extended human--AI collaboration. 
Several structural ideas and technical arguments emerged from exploratory interaction with AI-based reasoning systems (Gemini Deep Think (Google DeepMind) and ChatGPT Pro (OpenAI)), which were used at different stages of the conceptual development and technical verification of this work. 
The author has fully reworked and verified all arguments and bears sole responsibility for their correctness.


\begin{thebibliography}{9}

\bibitem{Cintioli2026}
P. Cintioli.
\textit{Rigid many-one degrees contain infinite antichains of 1-degrees}.
arXiv preprint, arXiv:2602.19960v3 [math.LO], 2026.

\bibitem{Maslova1979}
T.\,M.~Maslova.
\textit{Veroyatnostnye Metody i Kibernetika} (in Russian).
Kazan University, Kazan, vol.~15, pp.~51--60, 1979.

\bibitem{Odifreddi1981}
P.\,G.~Odifreddi,
\emph{Strong reducibilities},
Bull.\ Amer.\ Math.\ Soc.\ (N.S.) \textbf{4} (1981), no.~1, 37--86.
doi:\,10.1090/S0273-0979-1981-14863-1.

\bibitem{Odifreddi1999}
P.\,G.~Odifreddi,
\emph{Reducibilities},
in \emph{Handbook of Computability Theory},
E.\,R.~Griffor (ed.),
Studies in Logic and the Foundations of Mathematics, vol.~140,
Elsevier/North-Holland, Amsterdam, 1999, pp.~89--119.
doi:\,10.1016/S0049-237X(99)80019-6.


\bibitem{OdifreddiCRT}
P.~G.~Odifreddi,
\emph{Classical Recursion Theory},
Studies in Logic and the Foundations of Mathematics, vol.~125,
North-Holland, Amsterdam, 1989.


\bibitem{RSZ2026}
L. Richter, F. Stephan, X. Zhang.
\textit{Chains and Antichains Inside Many-One Degrees and Variants}.
Submitted preprint, available at \url{https://linus-richter.github.io}, 2026.


\bibitem{Soare1987}
R.~I.~Soare,
\emph{Recursively Enumerable Sets and Degrees: A Study of Computable Functions and Computably Generated Sets},
Perspectives in Mathematical Logic,
Springer-Verlag, Berlin, 1987.

\end{thebibliography}
\end{document}